\documentclass[11pt]{amsart}
\usepackage{graphicx}
\usepackage[active]{srcltx}
 \makeatletter
\renewcommand*\subjclass[2][2000]{%
  \def\@subjclass{#2}%
  \@ifundefined{subjclassname@#1}{%
    \ClassWarning{\@classname}{Unknown edition (#1) of Mathematics
      Subject Classification; using '1991'.}%
  }{%
    \@xp\let\@xp\subjclassname\csname subjclassname@#1\endcsname
  }%
}
 \makeatother
\usepackage{enumerate,url,amssymb,  mathrsfs}%, pdfsync}% ,showkeys, pdfsync}

\newtheorem{theorem}{Theorem}[section]
\newtheorem{lemma}[theorem]{Lemma}
\newtheorem*{lemma*}{Lemma}
\newtheorem{proposition}[theorem]{Proposition}
\newtheorem{corollary}[theorem]{Corollary}

\theoremstyle{definition}

\theoremstyle{remark}
\newtheorem{remark}[theorem]{Remark}

\numberwithin{equation}{section}

\def\XXint#1#2#3{{\setbox0=\hbox{$#1{#2#3}{\int}$}
\vcenter{\hbox{$#2#3$}}\kern-.5\wd0}}

\def\le{\leqslant}
\def\ge{\geqslant}
\begin{document}

\title{Schwarz lemma for holomorphic mappings in the unit ball}

\author{David Kalaj}
\address{University of Montenegro, Faculty of Natural Sciences and
Mathematics, Cetinjski put b.b. 81000 Podgorica, Montenegro}
\email{davidkalaj@gmail.com}

%\email{saksman@mappi.helsinki.fi}
%\date{}
% \subjclass {Primary 30C55, Secondary 31C05}
%\dedicatory{This paper is dedicated to our authors.}
%\keywords{Planar harmonic mappings, Quasiconformal mapping,
%  Convex domains, Rado-Kneser-Choquet theorem}

\maketitle

%\def\thefootnote{}
%\footnotetext{
%\texttt{\tiny File:~\jobname .tex,
%          printed: \number\year-\number\month-\number\day,
%          \thehours.\ifnum\theminutes<10{0}\fi\theminutes}
%}
\makeatletter\def\thefootnote{\@arabic\c@footnote}\makeatother

\begin{abstract}
In this note we establish a Schwarz Pick type inequality for holomorphic mappings between unit balls $\mathbf{B}_n$ and $\mathbf{B}_m$ in corresponding complex spaces. We also prove a Schwarz Pick type inequality for pluri-harmonic functions.
\end{abstract}

\maketitle
%\tableofcontents

\section{Introduction}\label{intsec}

The  Fr\'echet derivative of a holomorphic mapping $f:\Omega\to \mathbf{C}^m$, where $\Omega\subset \mathbf{C}^n$,  is defined to be the unique linear map $A=f'(z):\mathbf{C}^n\to \mathbf{C}^m$ such that $f(z+h)=f(z)+ f'(z) h +O(|h|^2)$. The norm of such a map is defined by $$\|A\|=\sup_{|z|=1} |Az|.$$

The classical Schwarz lemma states that $|f(z)|\le |z|$ for every holomorphic mapping of the unit disk $\mathbf{B}_1\subset \mathbf{C}$ into itself satisfying the condition $f(0)=0$. This inequality implies the following inequality for the derivative \begin{equation}\label{as}|g'(z)|\le \frac{1-|g(z)|^2}{1-|z|^2}\end{equation} for every holomorphic mapping $g$ of the unit disk into itself. On the other hand if $n,m$ are two positive integers and $\mathbf{B}_n\subset \mathbf{C}^n$ is the unit ball, then every holomorphic mapping $f:\mathbf{B}_n\to \mathbf{B}_m$, with $f(0)=0$ satisfies the inequality $|f(z)|\le |z|$, but the counterpart of \eqref{as} in the space does not hold provided that $m\ge 2$ (see e.g. \cite{Pavlovic} and corresponding sharp inequality \eqref{ejte} below). However it holds for $m=1$, while  for $m\ge 2$ it holds its weaker form namely $1-|g(z)|^2$  should be replaced by $\sqrt{1-|g(z)|^2}$. This is proved in Theorem~\ref{theo1}, which is the main result of this paper. By using the case $m=1$, in Theorem~\ref{te2} we prove Schwarz Pick inequality type inequality for pluriharmonic function, which extends a corresponding result for real harmonic functions \cite{kalvuo}

\subsection{Automorphisms of the unit ball}
Let $P_a$ be the
orthogonal projection of $\mathbf{C}^n$ onto the subspace $[a]$ generated by $a$, and let $$Q=Q_a =
I - P_a$$ be the projection onto the orthogonal complement of $[a]$. To be
quite explicit, $P_0 = 0$ and  $P=P_a(z) =\frac{\left<z,a\right> a}{\left<a, a\right>}$. Put $s_a = (1 - |a|^2)^{1/2}$ and define
$$ \varphi_a(z) =\frac{a-P_a z-s_a Q_a z}{1-\left<z,a\right>}.$$

If $\Omega= \{z\in \mathbf{C}^n:\left<z,a\right>\neq 1\}$,  then $\varphi_a$  is holomorphic in $\Omega$. It is clear that $\overline{\mathbf{B}}\subset \Omega$ for $|a| < 1.$

\begin{proposition}\cite[Theorem~2.2.2]{rudin}
If $\varphi_a$ is defined as above then

a) $\varphi_a(0)=a$ and $\varphi_a(a)=0$.

b) $\varphi_a'(0)=-s^2 P - s Q.$

c) $\varphi_a'(a)=-\frac{1}{s^2}P-\frac{1}{s}Q.$

d) $\varphi_a$ is an involution: $\varphi_a(\varphi_a(z))=z.$

e) $\varphi_a$ is a biholomorphism of the closed unit ball onto itself.

\end{proposition}
\section{The main results}
\begin{theorem}[The main theorem]\label{theo1}
If $f$ is a holomorphic mapping of the unit ball $\mathbf{B}_n\subset\mathbf{C}^n$ into $\mathbf{B}_m\subset \mathbf{C}^m$, then for $m\ge 2$
\begin{equation}\label{merkur}\| f'(z)\|\le \frac{{\sqrt{1-|f(z)|^2}}}{1-|z|^2}, \ \ z\in \mathbf{B}_n\end{equation} and for $m=1$ we have that \begin{equation}\label{merkur1}\| f'(z)\|\le \frac{1-|f(z)|^2}{1-|z|^2}, \ \ z\in \mathbf{B}_n.\end{equation} The inequalities \eqref{merkur} and \eqref{merkur1} are sharp. In particular we have the following sharp inequalities
 \begin{equation}\label{ejte}\| f'(0)\|\le
\left\{
  \begin{array}{ll}
     \sqrt{1-|f(0)|^2}, & \hbox{if $m\ge 2$;} \\
   1-|f(0)|^2 , & \hbox{if $m=1$.}
  \end{array}
\right.
\end{equation}
\end{theorem}
We need the following lemma
\begin{lemma}\label{lemafor} 
Let  $n$ be a positive integer and let $M=M_a=-s^2 P_a - s Q_a$ and $N=N_a = -\frac{1}{s^2}P_a-\frac{1}{s}Q_a $, where $s=\sqrt{1-|a|^2}$.
Then $$\|M\|=\left\{
               \begin{array}{ll}
                 \sqrt{1-|a|^2}, & \hbox{if $n\ge 2$;} \\
                 1-|a|^2, & \hbox{if $n=1$}
               \end{array}
             \right.
$$ and

$$\|N\|=\frac{1}{1-|a|^2}.$$

\end{lemma}
\begin{proof}[Proof of Lemma~\ref{lemafor}]
We have $$s^2 Pz + s Q z =sz + (s^2-s)\frac{\left<z,a\right> a}{\left<a, a\right>}.$$ So
\[\begin{split}|s^2 Pz + s Q z|^2&=\left|sz + (s^2-s)\frac{\left<z,a\right> a}{\left<a, a\right>}\right|^2\\&=s^2|z|^2+(s^4-s^2)\frac{\left<z,a\right>^2}{\left<a,a\right>}\le s^2 |z|^2.\end{split}\]
For $z\bot a$ the previous inequality becomes an equality. It follows that  $\|M\|=\sqrt{1-|a|^2}.$ The case $n=1$ is trivial and in this case $Q_a\equiv 0$.
Further, we establish the norm of the operator $$N=N_a = -\frac{1}{s^2}P-\frac{1}{s}Q.$$ We have $$Nz= -\frac{1}{s^2} P z-\frac{1}{s} Q z= \left(-\frac{1}{s^2}+\frac{1}{s}\right) \frac{\left<z,a\right> a}{\left<a, a\right>}-\frac{1}{s} z$$
and so
$$|N z|^2 =\frac{1}{s^2}|z|^2+\left(\frac{1}{s^4}-\frac{1}{s^2}\right)\frac{\left<z,a\right>^2}{|a|^2}.$$ By choosing $z=a/|a|$, we obtain  that $\|N\|= \frac{1}{1-|a|^2}.$
\end{proof}

\begin{proof}[Proof of Theorem~\ref{theo1}]
Let $\varphi_a$ be an involutive automorphism of the unit ball $\mathbf{B}_n$ onto itself such that $\varphi_a(0)=a$ and let $b=f(a)$. Let $\varphi_b$ be an  involutive automorphism of the unit ball $\mathbf{B}_m$ onto itself such that $\varphi_b(b)=0$ and let $g=\varphi_b^{-1}\circ f \circ \varphi_a^{-1}$. Then $f(z) =\varphi_b \circ g \circ \varphi_a$ and so $$f'(a) = \varphi_b'(0) g'(0) \varphi_a'(a)=M_b g'(0) N_a.$$ Since $g$ maps the unit ball into itself and satisfies  $g(0)=0$, by \cite[Theorem~8.1.2]{rudin}, it follows that $|g'(0)|\le 1$. Since $\|A\cdot B\|\le \|A\| \|B\|$, according to Lemma~\ref{lemafor} we obtain \eqref{merkur}.
The sharpness is proved by the following example.
Let   $t\in(0,\pi/2)$ and define $f_t(z,w) = (z\sin t  ,\cos t  )$. Then $f_t:\mathbf{B}_2\to \mathbf{B}_2$. Moreover $|f_t'(0)|= \sin t$ and $|f_t(0)|=\cos t.$ So $| f_t'(0)|=\sqrt{1-|f_t(0)|^2}.$
\end{proof}
Assume that $m=1$ and let $a$ be holomorphic function of the unit ball $\mathbf{B}_n$ into $\mathbf{C}^m=\mathbf{C}$. 
Since $a'$ is $\mathbf{C}$ linear, we regard $a'(z)$  as a vector $(a_{z_1},\dots a_{z_n})$ from the space $\mathbf{C}^n$ and we will denote it by $\nabla a$.
\begin{theorem}\label{te2}
Let $f$ be a pluriharmonic function of the unit ball $\mathbf{B}_n$ into $(-1,1)$.
Then  the following sharp inequality holds
\begin{equation}\label{sha}|\nabla f(z)|\le
\frac{4}{\pi}\frac{1-|f(z)|^2}{1-|z|^2},\ \ \ |z|<1.\end{equation}
\end{theorem}
\begin{proof}[Proof of Theorem~\ref{te2}] Let  $h$ be the pluri-harmonic
conjugate of $f$, i.e. assume that $a=f+ih$ is a holomorphic mapping. Then $a$  maps the unit ball $\mathbf{B}_n $
into the vertical strip $S= \{w: -1< \Re w< 1 \}\,.$

Let 
$$ g(z) = \frac{2i}{\pi}\log \frac{1+z}{1-z}.$$ Then 
$g$ is a conformal mapping of  the unit disk $ \mathbf U$ onto the strip $S$. Hence $b(z) = g^{-1} (a(z))$ is a holomorphic mapping of the unit ball onto the unit disk $\mathbf{U}$. Then we have that 
$$a(z) =\frac{2i}{\pi}\log \frac{1+b(z)}{1-b(z)}.$$ By \eqref{merkur1} we have
$$|b'(z)|\le \frac{1-|b(z)|^2}{1-|z|^2}.$$ On the other hand $$a'(z) =
\frac{4i}{\pi}\frac{b'(z)}{1-b(z)^2}.$$  Since $a$ is holomorphic we obtain that $$|a'(z)|=\sqrt{\sum_{k=1}^n |a_{z_k}|^2}=\sqrt{\sum_{k=1}^nf_{x_k}^2+\sum_{i=k}^nf_{y_k}^2}$$ and $$|\nabla f|=|(f_{x_1},f_{y_1},\dots, f_{x_n},f_{y_n})|=\sqrt{\sum_{k=1}^nf_{x_k}^2+\sum_{k=1}^nf_{y_k}^2}=|a'(z)|.$$
We will find the best possible constant $C$ such that
$$|\nabla f(z)|\le C \frac{1-|f(z)|^2}{1-|z|^2}.$$
As
$$|a'(z)|\le \frac{4}{\pi}\frac{1-|b(z)|^2}{|1-b(z)^2|} \, \frac{1}{1-|z|^2}$$
it will be enough to find the best possible constant $C$ such that
$$\frac{4}{\pi}\frac{1-|b(z)|^2}{|1-b(z)^2|} \, \frac{1}{1-|z|^2}\le C
\frac{1-|\Re a(z)|^2}{1-|z|^2}$$ or what is the same
$$\frac{4\pi}{({\pi^2-4|\mathrm{arg} \frac{1+b}{1-b}|^2})}\frac{1-|b|^2}{|1-b^2|}\le C\,, |b|<1\,.$$
Let $\omega= \frac{1+b}{1-b} = r e^{it}$. Then $-\pi/2\le t\le
\pi/2$ and
$$  1-|b|^2 = \frac{4 r  \cos  t}{r^2+ 2r \cos t +1},$$ $$
|1-b^2| = \frac{4 r |e^{it} |}{r^2+ 2r \cos t +1} ,$$ and hence the
last inequality with the constant $C= 4/\pi$ follows from
$$\frac{|\cos t|}{1-\frac{4}{\pi^2}t^2}\le 1$$
which holds for $t \in (-\pi/2,\pi/2) \, .$ This yields \eqref{sha}.

To show that the inequality \eqref{sha} is sharp, take the pluri-harmonic
function $$f(z) =
\frac{2}{\pi}\mathrm{arctan}\frac{2y_1}{1-x_1^2-y_1^2}.$$ It is easy to
see that $$|\nabla
f(0)|=\frac{4}{\pi}=\frac{4}{\pi}\frac{1-|f(0)|^2}{1-0^2}.$$
%Thus $C= \frac 4{\pi}$.
\end{proof}

If $ds=\frac{2|dz|}{1-|z|^2}$ is the hyperbolic metric in the unit ball, denote by $d_h$ the corresponding distance function. Now from \eqref{merkur1} we infer
\begin{corollary}
If $f:\mathbf{B}_n\to \mathbf{B}_1$ is holomorphic, then  $$d_h(f(z),f(w))\le  d_h(z,w).$$
\end{corollary}
For general $m$ we have 
\begin{corollary}
If $f:\mathbf{B}_n\to \mathbf{B}_m$ is holomorphic, then  $$\arcsin(|f(z)|)\le \mathrm{arctanh}(|z|).$$
\end{corollary}
By using scaling argument we obtain
\begin{corollary}
If $f$ is a holomorphic mapping of the unit ball $\mathbf{B}_n\subset\mathbf{C}^n$ into $\mathbf{C}^m$, then
\begin{equation}\label{hane}\| f'(z)\|\le \left\{
                                            \begin{array}{ll}
                                              \frac{\sqrt{\|f\|^2-|f(z)|^2}}{1-|z|^2}, \ \ z\in \mathbf{B}_n, & \hbox{if $m\ge 2$;} \\
                                              \frac{\|f\|^2-|f(z)|^2}{\|f\|(1-|z|^2)}, & \hbox{if $m=1$.}
                                            \end{array}
                                          \right.
\end{equation} Here $\|f\|:=\sup_z|f(z)|.$
\end{corollary}

In order to state a new corollary of the main result, recall the definition of  Bloch space $\mathcal{B}$ of holomorphic  mappings of the unit ball $\mathbf{B}_n$ into $\mathbf{C}^m$. We say that $f\in \mathcal{B}$ provided its seminorm satisfies  $\|f\|_{\mathcal{B}}=\sup_{|z|<1}(1-|z|^2) |f'(z)|<\infty$. Let $\mathcal{B}_1$ be the unit ball of $\mathcal{B}$. Let  $\mathfrak{B}$ be the set of bounded holomorphic mappings between $B_n$ and $\mathbf{C}^m$, i.e. of mappings satisfying the inequality $\|f\|=\sup_z|f(z)|<\infty$.
\begin{corollary}
The inclusion operator $\mathcal{I}:f\mapsto f$ between $\mathfrak{B}$ and $\mathcal{B}$ has norm equal to $1$.
\end{corollary}
\begin{proof}
It is clear from \eqref{hane} that $|\mathcal{I}|\le 1$. Prove the equality statement. Assume as we may that $n=m=2$. Let $f_0(z,w)=(  z,0  )$. Then $$\|f_0\|_{\mathcal{B}}=\sup_{|z|^2+|w|^2\le 1} (1-|z|^2-|w|^2)\|f'_0(z,w)\|=1,$$ and $$\|f_0\|=\sup_{|z|^2+|w|^2\le 1}\sqrt{|z|^2 }=1.$$ This finishes the proof.
\end{proof}
\begin{remark}
The inclusion operator is a restriction of a Bergman projection $\mathcal{P}_\alpha$, $\alpha>-1$, between $L^\infty(B_n)$ and $\mathcal{B}$ whose norm is greater than $1$. See corresponding results for the plane \cite{perala} and for the space \cite{kalmar}.
\end{remark}

\subsection{Acknowledgement} I would like to thank Dr Marijan Markovic for some remarks concerning pluriharmonic functions.


\begin{thebibliography}{1}

\bibitem{kalmar}
D. Kalaj, M. Markovic: \emph{Norm of the Bergman projection}, Mathematica Scandinavica, vol. 115, no. 1, pp. 143-160, 2014.

\bibitem{kalvuo}
D. Kalaj, M. Vuorinen,\emph{On harmonic functions and the Schwarz lemma.}
Proc. Am. Math. Soc. 140, No. 1, 161-165 (2012).
\bibitem{Pavlovic}
M. Pavlovi\'c:  \emph{A Schwarz lemma for the modulus of a vector-valued analytic function.} Proc Amer Math Soc,
2011, 139(3): 969-973
\bibitem{perala}
A. Per\"al\"a, \emph{Bloch space and the norm of the Bergman projection.}
Ann. Acad. Sci. Fenn., Math. 38, No. 2, 849-853 (2013).
\bibitem{rudin}
W. Rudin: \emph{Function theory of the unit ball in $\mathbf{C}^n$}, Springer, 1980.


\end{thebibliography}
\end{document}